\documentclass[12pt]{article}
\usepackage[english]{babel}
\usepackage{amsmath,amssymb,amsfonts,amscd}
\usepackage{amsthm}
\newtheorem {theorem}{Theorem}
\newtheorem{proposition}[theorem]{Proposition}
\newtheorem{lemma}[theorem]{Lemma}
\newtheorem{exemple}[theorem]{Example}

\newcommand{\R}{\mathbb R}

\newcommand{\w}{^W\!\!D}
\newcommand{\wu}{^W\!\!U}
\begin{document}

\title{Dirichlet problem associated with Dunkl Laplacian  on $W$-invariant open sets}
\author{
Mohamed Ben Chrouda\\
{\small Department of Mathematics, High Institute of Informatics and Mathematics}\\  {\small 5000 Monastir, Tunisia}\\
{\small E-mail: benchrouda.ahmed@gmail.com} \medskip\\
and \medskip \\
Khalifa El Mabrouk\\
{\small Department of Mathematics, High School of Sciences and Technology} \\ {\small 4011 Hammam Sousse, Tunisia}
\\
{\small E-mail: khalifa.elmabrouk@fsm.rnu.tn}
}
\date{}
\maketitle
\begin{abstract}
 Combining probabilistic  and analytic tools from potential theory, we investigate  Dirichlet problems associated with the Dunkl Laplacian $\Delta_k$. We establish, under some  conditions on the open set $D\subset\R^d$,  the existence of a unique continuous function $h$ in the closure of $D$, twice differentiable in $D$, such that
$$
\Delta_kh=0 \quad\textrm{in}\;D\quad\textrm{and}\quad h=f\quad\textrm{on}\; \partial D.
$$
We also give a probabilistic  formula characterizing  the solution $h$. The function $f$ is assumed to be continuous on the Euclidean boundary $\partial D$ of $D$.
\end{abstract}

\section{Introduction}

In their  monograph~\cite{hansen}, J.~Bliedtner and W. Hansen developed   four descriptions of potential theory using   balayage spaces, families of harmonic kernels, sub-Markov semigroups and  Markov processes. They proved that all these descriptions are  equivalent and gave a straight presentation of balayage theory which is, in particular,  applied to the generalized Dirichlet problem associated with a large class of differential and pseudo-differential operators.

Let $W$ be a finite reflection group on $\R^d$, $d\geq 1$, with root system $R$ and we fix a positive subsystem $R_+$ of $R$ and a nonnegative multiplicity function $k:R\to\R_+$.
For every $\alpha\in R$, let $H_\alpha$ be the hyperplane  orthogonal to $\alpha$ and $\sigma_\alpha$ be the reflection with respect to~$H_\alpha$, that is, for every~$x\in\R^d$,
$$
\sigma_\alpha x=x-2 \frac{\langle x,\alpha\rangle}{|\alpha|^2}\alpha
$$
where $\langle\cdot,\cdot\rangle$ denotes the Euclidean inner product of $\R^d$. C.~F.~Dunkl introduced in \cite{dunkl} the operator
 $$
 \Delta_k=\sum_{i=1}^dT_i^2,
 $$
 which will be called  later \emph{Dunkl Laplacian},
 where, for $1\leq i\leq d$, $T_i$ is the differential-difference operator defined for $f\in C^1(\R^d)$ by
$$T_i f(x)=\frac{\partial f}{\partial x_i}(x)+\sum_{\alpha\in R_+}k(\alpha)\alpha_i\frac{f(x)-f(\sigma_\alpha x)}{\langle\alpha,x\rangle}.$$

Our main goal in this  paper is to investigate  the \emph{Dirichlet problem} associated with the Dunkl Laplacian.
 More precisely,  given a bounded open set $D\subset\R^d$ and  a continuous real-valued function $f$ on $D^c:=\R^d\setminus D$, we are concerned  with   the  following problem:
\begin{equation}\label{dirp}
\displaystyle\left\{
\begin{array}{rcll}
 \Delta_kh & = & 0 & \mbox{in }D, \\
  h & = & f& \mbox{on }D^c.
\end{array}
\right.
\end{equation}
We mean by a solution of~(\ref{dirp}) every function $h:\R^d\to\R$ which is continuous in $\R^d$,  twice differentiable in~$D$ and such that both equations in~(\ref{dirp}) are pointwise fulfilled.  In the particular case where $D$ is the unit ball of $\R^d$, M.~Maslouhi and E.~H.~Youssfi~\cite{masl} solved problem (\ref{dirp})
  by methods from harmonic analysis using the Poisson kernel for $\Delta_k$  which is introduced by C.~F.~Dunkl and Y.~Xu~\cite{dunklxu}.
  It should be noted that, for balls with center $a\not=0$, the Poisson kernel for $\Delta_k$ is not known up to now.

Let us briefly introduce our approach. It is well known (see \cite{gy} and references therein) that there exists a c\`adl\`ag $\R^d$-valued Markov process $$X=(\Omega,{\mathcal{F}},{\mathcal{F}}_t,X_t,P^x),$$ which is called  \emph{Dunkl process},  with infinitesimal generator $\frac12\Delta_k$. For a given bounded Borel function $h:\R^d\to\R$, we define
$$
H_U h(x)=E^x[h(X_{\tau_U})]
$$
for every $x\in\R^d$ and every bounded open subset $U$ of $\R^d$, where
$$
 \tau_U=\inf\{t>0; X_t\notin U\}
 $$
 denotes the first exit time from $U$ by~$X$. We first show  that if $h$ is continuous in~$\R^d$ and twice differentiable in~$D$ then $\Delta_kh=0$ in $D$ if and only if $h$ is $X$-harmonic in~$D$, i.e., $ H_Uh(x)=h(x)$
 for every open set $U$ such that $\overline{U}\subset D$ (we shall write $U\Subset D$) and for every $x\in U$. We then conclude, using the general framework of balayage spaces \cite{hansen},  that problem~(\ref{dirp}) admits at most one solution. Moreover, if the open set~$D$ is regular for the Dunkl process, then  $H_D f$ will be the solution of~(\ref{dirp}) provided it is of class $C^2$ in~$D$.

For some examples of Markov processes, namely Brownian motion or $\alpha$-stable process, some additional geometric assumptions on  the Euclidean boundary $\partial D$ of $D$ permit a decision on the regularity of $D$. In fact, it is well known that $D$ is regular, with respect to  Brownian motion or $\alpha$-stable process, whenever  each  boundary point of $D$ satisfies  the "cone condition". For a particular choice of the root system $R$, we shall prove in Section~3 that the  cone condition is still sufficient for the regularity of~$D$ with respect to the Dunkl process. However, we could not  know whether this result holds   true for arbitrary root systems. In this  setting, we  only show that balls  of center $0$ are regular.

Finally, assuming that $D$ is regular, the study of problem~(\ref{dirp}) is equivalent to the study of smoothness of~$H_Df$. Indeed, as was mentioned above,~(\ref{dirp}) has a solution if and only if $$H_Df\in C^2(D).$$ To that end, we need to assume that $D$ is \emph{$W$-invariant} which means that $\sigma_\alpha(D)\subset D$ for every $\alpha\in R$. Hence, using the fact that the operator  $\Delta_k$ is hypoelliptic in $D$ (see \cite{kk1,mt2}) we prove that $H_Df$ is infinitely differentiable in $D$. Thus, we  not only deduce the existence and uniqueness of the solution to
\begin{equation}\label{dirp2}
\displaystyle\left\{
\begin{array}{rcll}
 \Delta_kh & = & 0 & \mbox{in }D, \\
  h & = & f& \mbox{on }\partial D,
\end{array}
\right.
\end{equation}
but we also prove that $h$ is given by the formula $h(x)=E^x[f(X_{\tau_D})]$.

 Throughout this paper, let $\lambda=\gamma+\frac d2-1$ and  assume that $\lambda >0$.
\section{Harmonic Kernels}
For the sake of simplicity, we assume in all the following that $|\alpha|^2=2$ for every $\alpha\in R$.
It follows from \cite{dunkl} that, for $f\in C^2(\R^d)$,
\begin{equation}\label{lapd}
\Delta_kf(x)=\Delta f(x)+2\sum_{\alpha\in R_+}k(\alpha)\left(\frac{\langle\nabla f(x),\alpha\rangle}{\langle\alpha,x\rangle}-\frac{f(x)-f(\sigma_\alpha(x))}{\langle\alpha,x\rangle^2}\right),
\end{equation}
where $\Delta$ denotes the usual Laplacian on $\R^d$. M. R\"osler has shown in \cite{rosler2} that $\frac12\Delta_k$ generates a Feller semigroup $P_t^k(x,dy)= p_t^k(x,y)w_k(y)dy$ which has the expression
\begin{equation}\label{chad}
p_t^k(x,y)=\frac{1}{c_kt^{\gamma+\frac{d}{2}}}
\exp\left(-\frac{|x|^2+|y|^2}{2t}\right)E_k\left(\frac{x}{\sqrt{t}},\frac{y}{\sqrt{t}}\right),
\end{equation}
where $E_k(\cdot,\cdot)$ is the Dunkl kernel associated with $W$ and $k$ (see \cite{dunklxu}), the constant $c_k$ is taken such that $P_1^k1\equiv1$, $\;\gamma=\sum_{\alpha\in
R_+}k(\alpha)$ and $w_k$ is the $W$-invariant weight function defined on~$\R^d$ by
$$
w_k(y)=\prod_{\alpha\in
R_+}|\langle y,\alpha\rangle|^{2k(\alpha)}.
$$
Let $X=(\Omega,{\mathcal{F}},{\mathcal{F}}_t,X_t,P^x)$ be the Dunkl process in $\R^d$ with transition kernel $P_t^k(x,dy)$.
For every bounded open subset $D$ of $\R^d$, let $\tau_D$ be the first  exit time  from $D$ by $X$. A point $z\in\partial D$ is said to be \emph{regular} (for $D$) if $P^z[\tau_D=0]=1$ and \emph{irregular}  if $P^z[\tau_D=0]=0$.
Notice that by Blumenthal's zero-one law, each boundary point of $D$ is either regular or irregular. It is also easy verified that the fact that  Dunkl process has right continuous paths yields that $P^x[\tau_D=0]=0$ if  $x\in D$ and $P^x[\tau_D=0]=1$ if  $x\in \R^d\setminus\overline D$.

\begin{proposition}\label{taus}
 $E^x[\tau_D]<\infty$ for  every $x\in \R^d$ and every bounded open subset $D$ of $\R^d$.
\end{proposition}
\begin{proof} Let $D$ be a bounded open subset of $\R^d$, $x\in \R^d$ and choose $r>0$ such that the ball $B=B(0,r)$ contains  $x$ and  $D$. Then, applying Fubini's theorem and using spherical coordinates,
\begin{eqnarray*}
  E^x[\tau_B]
   &\leq&  \int_0^{\infty}E^x[\mathbf{1}_B(X_s)]ds\\
   &=&\int_0^rt^{2\lambda+1}\int_0^\infty\int_{S^{d-1}} p_s^k(x,tz)w_k(z)\sigma(dz) ds\, dt.
   \end{eqnarray*}
Here and in all the following,  $\sigma$ denotes the surface area measure on the unit sphere $S^{d-1}$ of $\R^d$. It is well known (see \cite{rosler2,rosler3}) that for every $x,y\in \R^d$ and $s>0$,
$$
p_s^k(x,y)=\frac{1}{c_k^2}\int_{\R^d}e^{-\frac{s}{2}|\xi|^2}E_k(-ix,\xi)E_k(iy,\xi)w_k(\xi)d\xi
$$
 and
 $$
 \int_{S^{d-1}}E_k(ix,\xi)w_k(\xi)\sigma(d\xi)=\frac{c_k}{2^\lambda\Gamma(\lambda+1)}j_\lambda(|x|),
 $$
 where
$$
j_\lambda(z):=\Gamma(\lambda+1)\sum_{n=0}^\infty\frac{(-1)^nz^{2n}}{4^nn!\Gamma(n+\lambda+1)}
$$
is the Bessel normalized function.
 Hence
\begin{eqnarray}
  E^x[\tau_D] \nonumber &\leq& \int_0^{\infty}E^x[\mathbf{1}_B(X_s)]ds \nonumber\\
  &=&
  \frac{1}{2^{2\lambda-1}(\Gamma(\lambda+1))^2}\int_0^rt^{2\lambda+1}\int_0^\infty j_\lambda(ut)j_\lambda(u|x|)u^{2\lambda-1}dudt \nonumber \\
     &= &
   \frac{2^{2\lambda-1}\Gamma(\lambda+1)\Gamma(\lambda)}{2^{2\lambda-1}(\Gamma(\lambda+1))^2}\int_0^rt^{2\lambda+1}(\max(t,|x|))^{-2\lambda}dt\label{locb1}\\
   &=& \frac{r^2}{2\lambda}-\frac{|x|^2}{2\lambda+2}<\infty.\label{locb}
\end{eqnarray}
In order to get~(\ref{locb1}) one should think about  formula (11.4.33)  in \cite{abram}.
\end{proof}

Let $D$ be a bounded open subset of $\R^d$. For every $x\in \R^d$, the  exit distribution $H_D(x,\cdot)$  from $D$ by the Dunkl process starting at $x$ will be called {\em harmonic measure} relative to $x$ and $D$. That is, for every Borel subset $A$ of $\R^d$,
$$
H_D(x,A)=P^x(X_{\tau_D}\in A).
$$
  It is clear that $H_D(x,\cdot)=\delta_x$ the Dirac measure at $x$ whenever  $x\in \partial D$ is regular or $x\not\in \overline D$. We define
  $$
\w:=\cup_{w\in W}w(D)\quad \mbox{and}\quad \Gamma_D:=\overline{\w}\setminus D.
$$
In other words, $\w$ is the smallest open set  containing $D$ which is invariant under the reflection group $W$. The following theorem ensures that $H_D(x,\cdot)$ is supported by $\Gamma_D$ for every $x\in \overline{D}$.

\begin{theorem}\label{support}
Let $D$ be a bounded open subset of $\R^d$. Then for every $x\in \overline{D}$,
\begin{equation}\label{gmd}
P^x\left(X_{\tau_D}\in\Gamma_D\right)=1.
\end{equation}
\end{theorem}
\begin{proof} It is easily seen that for every regular boundary point $x$,  $P^x(X_{\tau_D}\in\Gamma_D)=\delta_x(\Gamma_D)=1$. Now, assume that $x\in D$ or $x\in\partial D$ is irregular and consider the function $\digamma$ defined for every $y,z\in\R^d$ by $\digamma(y,z)=0$ if $z\in\{\sigma_\alpha y;\alpha\in R_+\}$ and $\digamma(y,z)=1$ otherwise. Let
 $$
 Y_t:=\sum_{s<t}\mathbf{1}_{\{X_{s^-}\neq X_s\}}\digamma(X_{s^-},X_s), \quad t>0.
 $$
 It follows from \cite[Proposition 3.2]{gy} that for every $t>0$, $P^x(Y_t=0)=1$ and consequently
  $$
  P^x\left(\mathbf{1}_{\{X_{s^-}\neq X_s\}}\digamma(X_{s^-},X_s)=0;\forall s>0\right)=1.
  $$
  Then, since  $P^x(0<\tau_D<\infty)=1$ we deduce that
  $$
 P^x\left(\mathbf{1}_{\{X_{\tau_D^-}\neq X_{\tau_D}\}} \digamma(X_{\tau_D^-},X_{\tau_D})=0\right)=1.
  $$
  On the other hand, seeing that $X_{\tau_D^-}\in \overline D$ on $\{0<\tau_D<\infty\}$ we have
  $$
  \left\{X_{\tau_D}\not\in \Gamma_D, 0<\tau_D<\infty\right\}\subset
  \left\{ \mathbf{1}_{\{X_{\tau_D^-}\neq X_{\tau_D}\}} \digamma(X_{\tau_D^-},X_{\tau_D})=1  \right\}.
  $$
  This finishes the proof.
 \end{proof}

 Let $\mathcal{O}$ be the set of all bounded open subsets of $\R^d$. In the following, we denote by $\mathcal{B}_b(\R^d)$  the set of all bounded Borel measurable functions on $\R^d$. For every $D\in \mathcal{O}$ and $f\in \mathcal{B}_b(\R^d)$, let  $H_Df$ be the function defined on $\R^d$ by
 $$
 H_Df(x)= E^x\left[f(X_{\tau_D})\right]=
\int f(y) H_D(x,dy).
$$
Since $X$ is a Hunt process, it follows from the general framework of balayage spaces studied by   J. Bliedtner and W. Hansen  in \cite{hansen} that, for every $D\in \mathcal{O}$ and  $f\in \mathcal{B}_b(\R^d)$ with compact support,  $H_Df$ is continuous in $D$ and  for every $V\Subset D$,
\begin{equation}\label{vd}
H_VH_D=H_D \quad\textrm{in}\;\; V.
\end{equation}
Since $\textrm{supp}\,H_D(x,\cdot)\subset \Gamma_D$ for every $x\in \overline{\w}$, it is trivial that
$$
H_Df(x)= H_D\left(1_{\Gamma_D}f\right)(x),\quad x\in \overline{\w}.
$$
Hence, we immediately conclude that $H_Df$ is continuous in $D$. For every $D\in \mathcal{O}$ and  every $f\in \mathcal{B}_b(\Gamma_D)$, it will be convenient to denote again
\begin{equation}\label{hd}
 H_Df(x)=
\int f(y) H_D(x,dy),\quad x\in \overline{\w}.
\end{equation}
Let $U$ be an open subset of $\R^d$. A locally bounded function $h:^W\!\!\!\!U\rightarrow\R$ is said to be \emph{$X$-harmonic} in $U$ if $H_Dh(x)=h(x)$  for every open set $D\Subset U$ and every $x\in D$. If $U$ is bounded and $h$ is continuous in $\overline{^W\!U}$ then $h$ is $X$-harmonic in $U$ if and only if for every $x\in U$,
\begin{equation}\label{max}
h(x)=H_Uh(x).
\end{equation}
In fact, let $x\in U$ and let $(U_n)_{n\geq 1}$ be a sequence of nonempty bounded open subsets of $\R^d$ such that $x\in  U_n\Subset U_{n+1}$ and $U=\cup_nU_n$. Then  $(\tau_{U_n})_n$ converges to $\tau_U$ almost surely. Hence, the continuity of $h$ on $\overline{^W\!U}$ together with the quasi-left-continuity of the Dunkl process yield that $H_Uh(x)=\lim_nH_{U_n}h(x)$. The following proposition follows immediately from (\ref{max}).
\begin{proposition}\label{primax}
Let $U\in \mathcal{O}$  and let $h$ be a continuous function on $\overline{^W\!U}$. If $h$ is $X$-harmonic in $U$, then
$$\max_{x\in \overline{^W\!\!U}}h(x)=\max_{x\in \Gamma_U}h(x)\quad \textrm{and}\quad \min_{x\in \overline{^W\!\!U}}h(x)=\min_{x\in \Gamma_U}h(x).$$
\end{proposition}
We shall denote by $G^k$ the Green function of $\Delta_k$ which is defined for every $x,y\in \R^d$ by
$$
G^k(x,y)=\int_0^\infty p_t^k(x,y)dt.
$$
Since $p_t^k$ is symmetric in $\R^d\times\R^d$, we obviously see that the Green function $G^k$ is also symmetric in $\R^d\times\R^d$. Therefore, it follows from \cite[Theorem VI.1.16]{blum} that for every  $D\in \mathcal{O}$  and  for every $x,y\in \R^d$,
\begin{equation}\label{blum}
\int G^k(x,z)H_D(y,dz)=\int G^k(y,z)H_D(x,dz) .
\end{equation}
Furthermore, for every $y\in \R^d$, the function $G^k(\cdot,y)$ is excessive, that is, $G^k(\cdot,y)$ is lower semi-continuous in $\R^d$ and
$\int p_t^k(x,z)G^k(z,y)w_k(z)dz\leq G^k(x,y)$ for every $t>0$ and $x\in \R^d$. Consequently, it follows from \cite[Theorem IV.8.1]{hansen} that $G^k(\cdot,y)$ is hyperharmonic on $\R^d$, i.e., for every $D\in \mathcal{O}$ and for every $x\in \R^d$,
\begin{equation}\label{ghd}
\int G^k(z,y)H_D(x,dz)\leq G^k(x,y).
\end{equation}
\begin{lemma}
Let $f\in C^2_c(\R^d)$ and  $D\in \mathcal{O}$. For every $x\in \R^d$,
\begin{equation}\label{inv}
 \int G^k(x,y)\Delta_kf(y)w_k(y)dy=-2f(x).
\end{equation}
In particular,
\begin{equation}\label{dynk}
H_Df(x)-f(x)=\frac 12 E^x\left[\int_0^{\tau_D}\Delta_kf(X_s) ds\right].
\end{equation}
\end{lemma}
\begin{proof}
To get (\ref{inv}) it suffices to recall that
$$
\frac{\partial}{\partial t}P_t^k=\frac12 P_t^k\Delta_k,\quad t>0.
$$
Then, we integrate over $t$ and use the fact that $\lim_{t\rightarrow0}P_t^kf(x)=f(x)$ and $\lim_{t\rightarrow\infty}P_t^kf(x)=0$ for every  $x\in\R^d$.
Formula (\ref{dynk}) follows from (\ref{inv}) and the strong Markov property.
\end{proof}
Let $U$ be an open subset of $\R^d$. A function $h:^W\!\!\!\!U\rightarrow\R$ is said to be \emph{$\Delta_k$-harmonic} in $U$ if $h\in C^2(U)$ and $\Delta_kh(x)=0$ for every $x\in U$.

\begin{theorem}\label{fond1}
 Let $U$ be an open subset of $\R^d$ and let $h\in C(\wu)$. If $h\in C^2(U)$ then $h$ is $\Delta_k$-harmonic in $U$ if and only if $h$ is $X$-harmonic in $U$.
\end{theorem}
\begin{proof} Let $D\Subset U$ and let $x\in D$. Then
\begin{equation}\label{dyn}
   H_Dh(x)-h(x)=
   \frac 12 E^x\left[\int_0^{\tau_D}\Delta_k h(X_s)ds\right].
\end{equation}
In fact, choose an open set $V$ such that $D\Subset V\Subset U$, $f\in C^2_c(\R^d)$
which coincides with $h$ in $V$ and let $\psi=h-f$.
Then using (\ref{dynk}) we obtain
\begin{equation}\label{dyn1}
H_Dh(x)-h(x)
  = \frac 12 E^x\left[\int_0^{\tau_D}\Delta_kf(X_s)ds\right]+H_D\psi(x).
\end{equation}
For every $y\in\R^d$, let $N(y,dz)$ be the L\'evy kernel of the Dunkl process $X$ which is given by the following formula \cite{gy}
\begin{equation}\label{lvk}
N(y,dz)=
\sum_{\alpha\in R_+, \langle y,\alpha\rangle\neq0}\frac{k(\alpha)}{\langle\alpha,y\rangle^2}\delta_{\sigma_\alpha y}(dz).
\end{equation}
Since $\psi =0$ on $V$, it follows from \cite[Theorem 1]{iw} that
\begin{equation}\label{wat}
H_D\psi(x)=
E^x
 \left[\int_0^{\tau_D}\int \psi(z)N(X_s,dz)ds\right].
\end{equation}
On the other hand, by (\ref{lapd}) and (\ref{lvk})  we easily see that  for every $y\in D$,
\begin{equation}\label{wat1}
\Delta_kf(y)=\Delta_kh(y)-2\int \psi(z)N(y,dz).
\end{equation}
Thus formula~(\ref{dyn}) is obtained by combing (\ref{dyn1}), (\ref{wat}) and
(\ref{wat1}) above.
Now,  $h$ is obviously $X$-harmonic in $U$ whenever it is $\Delta_k$-harmonic in $U$.
Conversely, assume that $h$ is $X$-harmonic  in $U$ and let $x\in U$. Since $h\in C(\wu)\cap C^2(U)$ then $\Delta_kh$ is continuous in $U$ and consequently for every $\varepsilon>0$ there exists an open neighborhood $
D\Subset U$
of $x$ such that
$|\Delta_kh(y)-\Delta_kh(x)|\leq \varepsilon$ for every $y\in D$. Using formula (\ref{dyn}), we obtain
$$
|\Delta_kh(x)|=
 \frac{1}{E^x[\tau_D]}\left|E^x\left[\int_0^{\tau_D}\left(\Delta_kh(X_s)-\Delta_kh(x)\right)ds\right]\right|
\leq \varepsilon.
$$
Hence $\Delta_kh(x)=0$ as desired.
\end{proof}

\section{Regular Sets}
A bounded open subset  $D$ of $\R^d$ is said to be \emph{regular} if each $z\in \partial D$ is regular for $D$.
A complete  study of regularity is developed by J. Bliedtner and W. Hansen in \cite{hansen}. It follows that a point $z\in \partial D$ is regular for $D$ if and only if for every $f\in C(\Gamma_D)$,
$$
\lim_{x\in D,x\rightarrow z}H_D f(x)=f(z).
$$
Consequently, $H_Df$ is continuous on $^W\!\!\overline{D}$ whenever $D$ is regular and $f\in C(\Gamma_D)$.
 \begin{exemple}\rm
For all $R>r>0$, the ball $B(0,R)$ and the annulus $C(r,R)=\{x\in \R^d; \; r<\|x\|<R\}$ are regular.
 \end{exemple}
In fact, by \cite[Proposition VII.3.3]{hansen}, it is sufficient to find a neighborhood $V$ of $z\in \partial D$ and a real function $u$ such that
 \begin{itemize}
 \item [i)] $u$ is positive in $V\cap D$,
 \item [ii)] $u$ is $X$-harmonic in $V\cap D$,
 \item [iii)]$\lim_{x\in V\cap D,x\rightarrow z}u(x)=0$.
 \end{itemize}
Consider $V=\R^d\backslash\{0\}$ and $g$ the function defined on $V$ by
$$g(x)=\frac{1}{|x|^{2\lambda}}.$$
Using formula (\ref{lapd}), simple computation  shows that $g$ is $\Delta_k$-harmonic in $V$ which yields, by theorem \ref{fond1}, that $g$ is $X$-harmonic in $V$. Let $z\in \R^d$ such that $|z|=R$ and consider
$$u(x)=g(x)-\frac{1}{R^{2\lambda}},\quad x\in V.$$
It is clear that $u$ satisfy (i), (ii) and (iii) above with $D=B(0,R)$ or $D=C(r,R)$. Hence $z$ is regular for $D$.
Similarly, taking
$$u(x)=\frac{1}{r^{2\lambda}}-g(x),\quad x\in V,$$
we conclude that all points $z\in \R^d$ such that $|z|=r$ are regular for $C(r,R)$.
\\

A sufficient condition for regularity, known as the cone condition, is given in the following theorem for a particular root system $R$.
  \begin{theorem}
Let $(e_1,..., e_d)$ be the canonical basis of $\R^d$ and consider the root system  $R=\{\pm e_i,\;1\leq i\leq d\}$.
   Let $D$ be a bounded open subset of $\R^d$ and let $z\in \partial D$. Assume that there exists a cone  $C$ of vertex $z$  such that $C\cap B(z,r)\subset D^c$ for some  $r>0$. Then $z$ is regular for $D$.
  \end{theorem}
\begin{proof}
It is trivial that $P^z[\tau_D\leq t]\geq P^z[X_t\in C\cap B(z,r)]$ for all $t>0$. Therefore, in virtue of Blumenthal's zero-one low, it is sufficient to show that $\liminf_{t\rightarrow 0}P^z[X_t\in C\cap B(z,r)]$ is positive. Denote $C_0=C-z$, then
\begin{eqnarray}
 P^z[X_t\in C\cap B(z,r)]&=& \int_{C\cap B(z,r)}p_t^k(z,y)w_k(y)dy  \nonumber \\
 &=& \frac{1}{t^\gamma}\int_{C_0\cap B(0,\frac{r}{\sqrt{t}})}p_1^k(\frac{z}{\sqrt{t}},\frac{z}{\sqrt{t}}-y)w_k(z-\sqrt{t}y) dy.\label{reg}
\end{eqnarray}
It is trivial to see, from (\ref{chad}), that
$$
p_1^k(\frac{z}{\sqrt{t}},\frac{z}{\sqrt{t}}-y)= e^{-\frac{|y|^2}{2}}e^{-\langle\frac zt,z-\sqrt{t}y\rangle} E_k\left(\frac zt,z-\sqrt{t}y\right).
$$
Let  $k_i=k(e_i)$ and $y_i=\langle y,e_i\rangle$ for every $y\in\R^d$ and  $i\in \{1,...,d\}$.
It is known \cite{xu'} that for all $x,y\in \R^d$,
$$e^{-\langle x,y\rangle}E_k(x,y)=\prod_{i=1}^dM(k_i,2k_i+1,-2x_iy_i).$$
 $M(k_i,2k_i+1,\cdot)$ denotes the Kummer's function defined on $\R$ by
$$M(k_i,2k_i+1,s)=\sum_{n\geq 0}\frac{(k_i)_n}{(2k_i+1)_n}\frac{s^n}{n!}=1+\frac{k_i}{2k_i+1}s+\frac{k_i(k_i+1)}{(2k_i+1)(2k_i+2)}\frac{s^2}{2!}+ \cdots\quad.$$
Therefore, for any $y\in \R^d$ and $t>0$, we have
$$
\begin{array}{lll}
\displaystyle \frac{1}{t^\gamma}e^{-\langle\frac zt,z-\sqrt{t}y\rangle}E_k\left(\frac zt,z-\sqrt{t}y\right)w_k\left(z-\sqrt{t}y\right)\medskip\\
\displaystyle =\prod_{i=1}^d\frac{M\left(k_i,2k_i+1,-2\frac{z_i}{t}
(z_i-\sqrt{t}y_i)\right)(z_i-\sqrt{t}y_i)^{2k_i}}{t^{k_i}}.
\end{array}
$$
First, it is clear that
$$\frac{1}{t^{k_i}}M\left(k_i,2k_i+1,-2\frac{z_i}{t}(z_i-\sqrt{t}y_i)\right)(z_i-\sqrt{t}y_i)^{2k_i}=
\left\{
  \begin{array}{ll}
    1 & \textrm{if} \; k_i=0 \\
    y_i^{2k_i} & \textrm{if}\; z_i=0.
  \end{array}
\right.
$$
Next, assume that $k_i >0$ and $z_i\neq0$ for some $i\in \{1,\cdots,d\}$. Then, it follows from the integral representation of $M(k_i,2k_i+1,\cdot)$ that
\begin{eqnarray*}
&&\hspace*{-2cm}\frac{1}{t^{k_i}}M\left(k_i,2k_i+1,-2\frac{z_i}{t}(z_i-\sqrt{t}y_i)\right) \\
 &=&\frac{\Gamma(2k_i+1)}{\Gamma(k_i)\Gamma(k_i+1)}\int_0^1\frac{1}{t^{k_i}}e^{-2\frac{z_i}{t}(z_i-\sqrt{t}y_i)u}u^{k_i-1}(1-u)^{k_i}du\\
   &=& \frac{\Gamma(2k_i+1)}{\Gamma(k_i)\Gamma(k_i+1)}\int_0^{\frac 1t}e^{-2z_i(z_i-\sqrt{t}y_i)v}v^{k_i-1}(1-tv)^{k_i}dv.
\end{eqnarray*}
Now, applying the Lebesgue dominated convergence theorem, we obtain
$$
 \lim_{t\rightarrow0} \frac{1}{t^{k_i}}M\left(k_i,2k_i+1,-2\frac{z_i}{t}(z_i-\sqrt{t}y_i)\right) = \frac{\Gamma(k_i+1)}{\sqrt{\pi}z_i^{2k_i}}.
 $$
Thus
$$
\begin{array}{lll}
\displaystyle \lim_{t\rightarrow0}\frac{1}{t^\gamma}e^{-\langle\frac zt,z-\sqrt{t}y\rangle}E_k\left(\frac zt,z-\sqrt{t}y\right)w_k\left(z-\sqrt{t}y\right)\medskip\\
\displaystyle \geq \prod_{i=1}^d\min\left(1,y_i^{2k_i},\frac{\Gamma(k_i+1)}{\sqrt{\pi}}\right)=:\theta(y).
\end{array}
$$
Hence,  Fatou's lemma applied to (\ref{reg}) yields that
$$\liminf_{t\rightarrow 0}P^z[X_t\in C\cap B(z,r)]\geq \int_{C_0}e^{-\frac{|y|^2}{2}}\theta(y)dy>0.$$
\end{proof}
\section{ Dirichlet Problem}
 This section is devoted to study the following Dirichlet problem :
    Giving a regular open subset $D$ of $\R^d$ and a function $f\in C(\Gamma_D)$, we shall investigate existance and  uniqueness of  function $h\in C(\overline{\w})\cap C^2(D)$ satisfying the boundary value problem
 \begin{equation}\label{pbd}
 \left\{\begin{array}{rcll}
\Delta_k h&=&0&\;\textrm{in}\; D,\\
 h&=&f&\;\textrm{in}\; \Gamma_D.
\end{array}
\right.
 \end{equation}
 For every square integrable functions $\varphi$ and $\psi$ on $\R^d$ with respect to the measure $w_k(x)dx$, we define
 $$\langle\varphi,\psi\rangle_k=\int \varphi(x)\psi(x)w_k(x)dx.$$
\begin{lemma}
For every  bounded open set $D$ and
 for every $\varphi,\psi\in C^2_c(\R^d)$,
\begin{equation}\label{td}
    \langle H_D\psi, \Delta_k\varphi\rangle_k = \langle \Delta_k\psi,H_D\varphi\rangle_k.
\end{equation}
\end{lemma}
\begin{proof}
 Applying formula (\ref{inv}) to $\psi$, we have
\begin{equation}\label{psifi}
\langle H_D\psi, \Delta_k\varphi\rangle_k= -\frac12\int G^k(z,y)\Delta_k\psi(y)w_k(y)dyH_D(x,dz)\Delta_k\varphi(x)w_k(x)dx.
\end{equation}
Then (\ref{td}) is obtained by Fubini's theorem and formulas (\ref{blum}) and  (\ref{inv}). Here,
since $\varphi$ and $\psi$ are with compact supports, formulas (\ref{ghd}) and (\ref{locb}) justify the transformation of the integrals in (\ref{psifi}) by Fubini's theorem.
\end{proof}
 A  set $D$ is called \emph{$W$-invariant} if $\w=D$ which, in turn, is equivalent to $\Gamma_D=\partial D$. We finally have the necessary tools at our disposal for solving the following Dirichlet problem.
\begin{theorem}\label{aim}
Let $D$ be a $W$-invariant regular open subset of $\R^d$. For every function $f\in C(\partial D)$,  there exists one and only one function
 $h\in C(\overline{D})\cap C^2(D)$ such that
\begin{equation}\label{dir}
 \left\{\begin{array}{rcll}
\Delta_kh&=&0&\;\textrm{in}\; D,\\
 h&=&f&\;\textrm{in}\; \partial D.
\end{array}
\right.
 \end{equation}
Moreover, $h$ is given by
 $$
 h(x)=\int_{\partial D}f(y)H_D(x,dy),\quad x\in \overline{D}.
 $$
 \end{theorem}
 \begin{proof}
 In virtue of Theorem \ref{fond1}, we observe that for $f\in C(\partial D)$, every solution $h$ of  (\ref{pbd}) satisfies necessarily :
\begin{equation}\label{pbdx}
 \left\{\begin{array}{ll}
h\textrm{ is X-harmonic in}\; D,\\
 h=f\;\textrm{in}\; \partial D.
\end{array}
\right.
 \end{equation}
 Then, by Proposition \ref{primax},  (\ref{dir}) admits at most one solution. The function $H_Df$ is $X$-harmonic in $D$ by (\ref{vd}). Moreover, the regularity of $D$ yields that $H_Df$ is a continuous extension of $f$ to $\overline{D}$. Therefore,
 according to Theorem \ref{fond1}, $H_Df$ will be the unique solution of (\ref{dir}) provided  it is twice differentiable in $D$. On the other hand, it has been shown in \cite{kk1} that  $\Delta_k$ is hypoelliptic in $D$ (see also \cite{mt2}), i.e., a continuous function $g$ in $D$ which satisfies
\begin{equation}\label{hyp}
\langle g,\Delta_k\varphi\rangle_k=0 \quad \textrm{for all}\;\;\varphi\in C^\infty_c(D)
\end{equation}
is necessary infinitely differentiable in $D$. Thus to complete the proof we only need to show that (\ref{hyp}) holds true for $g=H_Df$. To this end
 let $\varphi\in C^\infty_c(D)$ and let $(f_n)_{n\geq 1}\subset~ C^2_c(\R^d)$ be a sequence which converges uniformly to $f$ in $\partial D $. Since $H_D\varphi(y)=0$ for all $y\in \R^d$, applying (\ref{td}) we obtain
\begin{equation}\label{disn}
\langle H_Df_n, \Delta_k\varphi\rangle_k = 0,\quad n\geq 1.
\end{equation}
On the other hand,
\begin{eqnarray*}
\sup_{x\in \overline{D}}|H_Df_n(x)-H_Df(x)|
\leq\sup_{y\in \partial D}|f_n(y)-f(y)|\longrightarrow 0\quad\textrm{as}\quad n\longrightarrow\infty.
\end{eqnarray*}
Hence $H_Df$ satisfies (\ref{hyp})  by letting $n$ tend to $\infty$ in~(\ref{disn}).
\end{proof}
It should be noted that the hypothesis "$D$ is $W$-invariant" is only needed to get the hypoellipticity of $\Delta_k$. For open set $D$ which is not $W$-invariant, the question whether  $\Delta_k$ is hypoelliptic in $D$ or not remained open. In the case of positive answer, analogous arguments as in the proof of Theorem \ref{aim} will immediately  imply  that $H_Df$ is the unique solution of problem (\ref{pbd}).

Let us notice that, using methods from harmonic analysis,  M. Maslouhi and E. H. Youssfi~\cite{masl} studied problem (\ref{dir}) in the special  case where $D=B$ is the unit ball of~$\R^d$. They proved that, for any  $f\in C(\partial B)$, the function $h$ given by
$$ h(x)=\int_{\partial B} P_\kappa(x,y)f(y)w_k(y)\sigma(dy),\; x\in B$$
is the unique solution of (\ref{dir}), where $P_\kappa$ denotes the Poisson kernel introduced by C. F. Dunkl and Y. Xu  \cite{dunklxu}. Hence, our above theorem  immediately yields that for every $x\in B$,
$$
H_B(x,dy)=P_\kappa(x,y)w_k(y)\sigma(dy).
$$

\end{document}